\documentclass[12pt]{article}
\usepackage{amsmath}
\usepackage{amsfonts}
\usepackage{amssymb}
\usepackage{amsthm}
\usepackage{stackrel}

\usepackage{color}

\newtheorem{thm}{Theorem}[section]
\newtheorem{co}{Corollary}[section]

\newcommand{\R}{{\rm I}\kern-0.18em{\rm R}}
\newcommand{\1}{{\rm 1}\kern-0.25em{\rm I}}
\newcommand{\E}{{\rm I}\kern-0.18em{\rm E}}
\newcommand{\p}{{\rm I}\kern-0.18em{\rm P}}
\makeatletter

\usepackage{epsfig}

\usepackage{fancyhdr}

\usepackage{graphics}

\usepackage{amssymb}
\usepackage{amsxtra}
\usepackage{amsfonts}

\usepackage{afterpage}
\usepackage{eucal}
\usepackage{eufrak}

\usepackage{enumerate}

\allowbreak

\author{Lev B Klebanov\footnote{Department of Probability and Statistics, MFF, Charles University, Prague-8, 18675, Czech Republic, e--mail: levbkl@gmail.com},  Irina V. Volchenkova\footnote{Department of Probability and Statistics, MFF, Charles University, Prague-8, 18675, Czech Republic, e--mail: i.v.volchenkova@gmail.com}, Ashot V. Kakosyan\footnote{Yerevan State University, Yerevan, Armenia}}
\title{ On a characterization of infinitely divisible distributions with Gaussian component}
\date{}

\begin{document}
\maketitle

\begin{abstract}
We give a necessary and sufficient condition for symmetric infinitely divisible distribution to have Gaussian component. The result can be applied to approximation the distribution of finite sums of random variables. Particularly, it  shows that for a large class of distributions with finite variance stable approximation appears to be better than Gaussian. 

\noindent
{\bf keywords}: infinitely divisible distributions; Gaussian component; approximations of sums of random variables.
\end{abstract}

\section{Formulation of the problem and main result}\label{s1}
\setcounter{equation}{0}

Let $f(t)$ be infinitely divisible characteristic function of a random variable $X$. As usual (see, for example, \cite{LO}), we say $X$ has Gaussian component if there are two independent random variables $Y_1,Y_2$ such that $X=Y_1+Y_2$, and $Y_1$ has non-degenerate Gaussian distribution. In opposite case we say, $f$ is without Gaussian component. Our aim is to give a characterization of infinitely divisible distributions with Gaussian component. 
\begin{thm}\label{th1}
Suppose that $f(t)$ is a symmetric infinitely divisible characteristic function. For any positive integer $m$ 
$f_m(t)=f^{1/m}(\sqrt{m}t)$ is a characteristic function. The following limit
\begin{equation}\label{eq1}
\lim_{m\to \infty}f_m(t) =g(t),
\end{equation}
exists. Here $g(t)$ is characteristic function of Gaussian or degenerate at zero distribution. The function $f(t)$ is without Gaussian component if and only if $g(t)=1$ for all $t$.\footnote{Let us note that we have no moment conditions in Theorem \ref{th1}}
\end{thm}
\begin{proof}Taking into account the fact of symmetry of $f(t)$ we can write Levy-Khinchine representation for it in the form
\[ \log f(t) = \int_{-\infty}^{\infty}\Bigl( \cos(tx)-1\Bigr)\frac{1+x^2}{x^2}d\theta (x)= \]
\[  =-\sigma^2 t^2 -4\int_{+0}^{\infty}\sin^2(tx/2)\frac{1+x^2}{x^2}d\theta (x), \]
where $\sigma^2 \geq 0$ is a jump of $\theta$ at zero. 

Let us show that $\int_{+0}^{\infty}\sin^2(tx/2)\frac{1+x^2}{x^2}d\theta (x)=o(t^2)$ as $t \to \infty$. For positive $t$ define $\varepsilon =\min (1,1/\sqrt{t})$. Really, for $t>1$ we have
\[  \int_{0}^{\infty}\sin^2(t x/2)\frac{1+x^2}{x^2}d\theta (x) = \]
\[ = t^2/2 \int_{0}^{\varepsilon}\Bigl(\frac{\sin (tx/2)}{tx/2}\Bigr)^2(1+x^2)d \theta (x)+
2 \int_{\varepsilon}^{\infty}\sin^2(tx/2) \frac{1+x^2}{x^2}d\theta (x) \leq \]
\[ \leq t^2/2(1+1/t)*(\theta (1/\sqrt{t})-\theta(0))+C (1+t) = o(t^2), \]
where $C>0$ is a constant. 

Now we see, that
\[ f_m(t) =\exp \Bigl( -\sigma^2 t^2 -\frac{1}{m} o(m*t^2) \Bigr) \stackrel[m \to \infty]{}{\longrightarrow} \exp \{-\sigma^2 t^2\}. \]
It is easy to see, that if $f(t)$ has no Gaussian component then $\sigma =0$, and $f_m(t) \to 1$ as $t \to \infty$.
\end{proof}

Now it is easy to obtain a characterization of Gaussian distribution.
\begin{co}\label{c1}In the conditions of Theorem \ref{th1} the variance of random variable with characteristic function $f(t)$ equals to that of limit variable with characteristic function $g(t)$ if and only if $f(t)$ is characteristic function of Gaussian distribution.
\end{co}
\begin{proof} The random variable $X$ with characteristic function $f(t)$ must has finite variance $\sigma^2$ because the limit distribution has such variance. However, $X=Y_1+Y_2$, where $Y_1$ has Gaussian distribution with variance $\sigma^2$, and therefore $Y_2$ has zero variance which means, that it has degenerate distribution.
\end{proof}
The next Corollary shows, that in condition of Theorem \ref{th1} we do not have convergence of variances.
\begin{co}\label{c2} Let $X$, $X_m$ be random variables with characteristic functions $f(t)$ and $f_m(t)$, correspondingly. Then $Var(X)=Var(X_m)$, $m=2,3, \ldots$. If $X$ has non-Gaussian distribution, then the variance of $X$ is strongly greater than variance of limit distribution with characteristic function $g(t)$.
\end{co}
\begin{proof}This statement is now obvious.
\end{proof}

Let us look at higher moments of the distribution supposing they exist. Denote by $X(m)$ random variable with characteristic function $f_m(t)$. As above, we suppose that characteristic function $f(t)$ is symmetric and infinitely divisible. Suppose that $\mu_4(m)=\E X^4(m)<\infty$. It is clear that $\mu_2(m)=\mu_2(1)=\E X^2(1)$ does not depend on $m$. However, calculating forth derivative of $f_m(t)$ at $t=0$, it is not difficult to find that
\begin{equation}\label{eq2}
\mu_4(m)-3\mu_2^2(m) = m (\mu_4(1)-3\mu_2(1).
\end{equation}
Equality (\ref{eq2}) may be rewritten in the form
\begin{equation}\label{eq3}
\kappa(m) =m \kappa(1),
\end{equation}
where $\kappa(m)=\mu_4(m)/\mu_2^2(m)-3$ is the kurtosis of $X(m)$.

\section{Positive infinitely divisible random variables}\label{s2}
\setcounter{equation}{0}

Here we give results similar in some sense to that from previous section. Namely, let us consider positive random variable $W$ with Laplace transform
\[ L(s)=\E \exp(-sW), \; \; s>0.\]
Suppose that $W$ has infinitely divisible distribution\footnote{Here it means that for any positive integer $n$ there are $n$ non-negative i.i.d. random variables $W_1, \ldots, W_n$ such that $W=W_1+ \ldots +W_n$}. It is known (see, for example, \cite{Ph}), that
\begin{equation}\label{eqL1}
L(s) = \exp\Bigl\{- \int_0^{\infty}\frac{1-e^{-a s}}{1-e^{-a}}d\mu(a)\Bigr\}, \; \; s>0,
\end{equation}
where $\mu$ is a Borel measure. There is one-to-one map between Laplace transforms $L(s)$ of infinitely divisible distributions and measures $\mu$. In a way similar to that of Section \ref{s1}, we consider Laplace transform
\[ L_m(s) =L^{1/m}(ms) \]
for positive integer $m$.

\begin{thm}\label{th2}
Suppose that $L(s)$ is a infinitely divisible Laplace transform of positive random variable $W$. For any positive integer $m$ 
$L_m(s)=L^{1/m}(m s)$ is a Laplace transform. The following limit
\begin{equation}\label{eqL2}
\lim_{m\to \infty}L_m(s) =D(s),
\end{equation}
exists. Here $D(s)$ is Laplace transform of degenerate distribution. $\p\{W<x\} >0$ for all $x>0$ if and only if $D(s)=1$ for all $s>0$.
\end{thm}
\begin{proof}
Let us write the integral in (\ref{eqL1}) as a sum of $\sigma s$ (where $\sigma$ is a jump of $\mu$ at zero) and 
\begin{equation}\label{eqL3}  
\int_{+0}^{\infty}\frac{1-e^{-a s}}{1-e^{-a}}d\mu(a). 
\end{equation}
Arguments very similar to used in the proof of Theorem \ref{th1} show that integral (\ref{eqL3}) is $o(s)$ as $s \to \infty$. 
To finish the proof it is enough to mention that $\exp\{-\sigma s\}$ is Laplace transform of the distribution, concentrated at the point $\sigma$.
\end{proof}

\section{Interpretation as return trip from Gaussian distribution}\label{s3}
\setcounter{equation}{0}

In the most simple variant of Central Limit Theorem one has the following statement. Let $\xi_1, \xi_2, \ldots $ be a sequence of independent identically distributed (i.i.d.) symmetric random variables with finite variance. Then the normalized sum $S_m =\frac{1}{\sqrt{m}}\sum_{j=1}^m \xi_j$ converges as $m \to \infty$ in distribution to Gaussian law with the same variance as that of $\xi_1$. The results above may be interpreted in the following way. We have characteristic function $f(t)$ of the sum $S_m$, supposing it is infinitely divisible. Then the characteristic function on $\xi_1=X(m)$ is $f_m(t)=f^{1/m}(\sqrt{m}t)$. In the case when $f(t)$ is not Gaussian characteristic function it is natural to assume, that $f_m(t)$ is far in therms of a distance from Gaussian characteristic function with the same variance. It is possible to give an estimate non-closeness to this Gaussian distribution. 

Let us formulate the direct Theorem for the convergence to Gaussian distribution. Let $\xi_1, \xi_2, \ldots $ be a sequence of independent identically distributed (i.i.d.) symmetric random variables with finite variance, and let $S_m =\frac{1}{\sqrt{m}}\sum_{j=1}^m \xi_j$ be their normalized sum. Define probability distance between two random variables $U$ and $V$ as 
\begin{equation}\label{eq4}
\lambda_r(U,V)=\sup_{t \in \R^1}\frac{|f_U(t)-f_V(t)|}{|t|^r},
\end{equation}
where $f_U(t)$ and $f_V(t)$ are characteristic functions of random variables $U$ and $V$ (for the definitions see \cite{Z}). Here we are considering $r>2$. It is known (see, for example, \cite{Kl}) that
\begin{equation}\label{eq5}
\lambda_r (S_m,Z) \leq m^{-(r/2-1)}\lambda_r(\xi_1,Z),
\end{equation}
where $Z$ is Gaussian random variable with the same variance as that of $\xi_1$.  Of course, we assume that $\lambda_r(\xi_1,Z) < \infty$.

The inequality (\ref{eq5}) may be rewritten in backward as
\begin{equation}\label{eq6}
\lambda_r(X(m), Z) \geq  m^{r/2-1}\lambda_r(X(1),Z). 
\end{equation}
The inequality (\ref{eq6}) shows that $X(m)$ and Gaussian $Z$ have to be far from each other for large values of $m$. 

This remark may be of some use in Finance. In \cite{KV} there was shown that the distributions of main financial indexes do not have heavy tails. However, their distributions are not Gaussian. Because 
corresponding random variables may be represented as sums of very large number of of i.i.d. random variables it can be approximated by infinitely divisible distributions. From previous remark we see, the distributions of summands are rather far from Gaussian. Let us show, the distribution of summands can be approximated by a stable distribution better than by Gaussian one just in the case when the summands distribution has finite second moment. If so, we can use central pre-limit theorem to show corresponding fact for the sums themselves (see, for example, \cite{Kl}). We will not show here that the summands distribution may be approximated by a stable distribution. Let us just mention this can be used in the same manner as in \cite{KV} for symmetrized gamma-distribution. 

The transformations from $f(t)$ to $f_m(t)=f^{1/m}(\sqrt{m}t)$ and from $L(s)$ to $L_m(s)=L^{1/m}(ms)$ give "historical" distribution of summands. In view of Theorems \ref{th1} and \ref{th2}, under some conditions, these distributions were born by degenerated at zero distribution. For the first look it seems to be paradoxical, however it is not. The reason is that in praktice one does not need to come to limit as $m$ tends to infinity.

\end{document}